\documentclass{article}
\usepackage{times, amsmath, amsthm, amssymb, graphicx,mathrsfs}
\input xy
\xyoption{all}
\theoremstyle{definition}

\newtheorem{thm}{Theorem}[subsection]
\newtheorem{cor}[thm]{Corollary}
\newtheorem{lem}[thm]{Lemma}
\newtheorem{pro}[thm]{Proposition}

\newtheorem{df}[thm]{Definition}
\newtheorem{rem}[thm]{Remark}

\newcommand{\mmC}{\mathscr{C}}
\newcommand{\mG}{\mathscr{G}}

\newcommand{\mO}{\mathcal{O}}

\newcommand{\mN}{\mathcal{N}}
\newcommand{\mM}{\mathcal{M}}
\newcommand{\mI}{\mathcal{I}}

\newcommand{\mL}{\mathcal{L}}

\newcommand{\bD}{\mathbb{D}}

\DeclareSymbolFont{AMSb}{U}{msb}{m}{n}
\DeclareMathSymbol{\boldk}{\mathord} {AMSb}{"7C}

\begin{document}
\title{\addtocounter{footnote}{1} The V-filtration for tame unit $F$-crystals}
\author{Theodore J. Stadnik, Jr. \footnote{Author partially supported by NSF grant DMS-0636646 \ \hspace*{800pt} MSC2010  \hspace*{13pt} MSC2010 Codes 16D10, 14B05 } \\
	Northwestern University \\}
\date{}
\maketitle

\begin{abstract}
    \centering
    \begin{minipage}{0.6\textwidth}
    Let $X$ be a smooth variety over an algebraically closed field of characteristic $p > 0$, $Z$ a smooth divisor, and $j: U=X \setminus Z \rightarrow X$ the natural inclusion.  An 
axiomatization of the properties of a $V$-filtration on a unit $F$-crystal is proposed and is proven to determine a unique filtration.  It is shown that if $\mM$ is a tame unit $F$-crystal 
on $U$ then such 
a $V$-filtration along $Z$ exists on $j_*\mM$.  The degree zero component of the associated graded module is proven to be the (unipotent) nearby cycles functor of Grothendieck and Deligne under the Emerton-Kisin Riemann-Hilbert correspondence.  A few applications to $\mathbb{A}^1$ and gluing are then discussed.
    \end{minipage}
\end{abstract}

\tableofcontents

\section{Introduction}

\hspace{10pt} An important construction for $\bD$-modules in characteristic zero is the $V$-filtration. While this filtration has many interesting applications, this paper mostly focuses on only a single aspect.  The $V$-filtration gives a purely algebraic method for discussing the nearby and vanishing cycles functors for $\bD$-modules without first passing through the Riemann-Hilbert correspondence.  This paper will prove the existence of a unique $V$-filtration in the case when $\mM$ is a tame unit $F$-crystal, a critical first step to a more intricate theory.  As in the characteristic zero setting, this $V$-filtration can be used to define a (unipotent) nearby cycles functor on the category of tame $F$-crystals.  This functor is compatible under the Emerton-Kisin Riemann-Hilbert correspondence to Grothendieck and Deligne's (unipotent) tame nearby cycles functor.  As an application, the unipotent nearby cycles functor will be combined with a naive vanishing cycles functor to give a gluing theorem for tame unit $F$-modules on $\mathbb{A}^1$.\\

\hspace{10pt} The $V$-filtration in characteristic zero was first studied by Kashiwara and Malgrange (\cite{Kas} and \cite{Mal}). Fixing a normal crossings divisor $f:X \rightarrow \mathbb{A}^1_{\mathbb{C}}$ inside of a smooth complex variety $X$, it was shown that there is a unique $\mathbb{C}$-indexed filtration on holonomic $\bD$-modules satisfying formal properties.  The four most important formal properties are related to coherence, multiplication by $f$ being an operator of degree one, $\partial_f$ being an operator of degree minus one, and the automorphism $f\partial_f-r$ acting on the $r^{th}$ component of the associated graded being nilpotent.  A consequence of the formal properties is that the degree zero component of the associated graded corresponds under the Riemann-Hilbert correspondence to the unipotent nearby cycles functor.  Taking different components of the associated graded, one can recover the nearby and vanishing cycles functors.  Moreover, on the graded objects the morphism ``multiplication by $f$'' will correspond to the variation map, $\partial_f$ the canonical map, and $f\partial_f$ the monodromy action.  \\

\hspace{10pt} Unfortunately, in positive characteristic the techniques from the classical setting do not translate well.  Most notably, there is not a single Euler operator but infinitely many Euler operators $f^{p^n}\partial_f^{[p^n]}$ which all have integral eigenvalues.  Additionally, the ring $\bD$ is not Noetherian, which impedes proving any statements which depend on finite generation.  The most general definition for nearby and vanishing cycles is given in \cite{SGA7II}, where it is defined on the derived category of constructible sheaves with coefficients in $\Gamma$.  Even though the definition does not require the residue characteristic of the base field be invertible in $\Gamma$, for many theorems this hypothesis is critical.  This was an impedance to defining the notion of a perverse sheaf for $p$ not invertible in $\Gamma$.  This impedance was removed independently by Gabber and Emerton-Kisin (\cite{Ga} and \cite{EK}).  Gabber showed that the category of constructible sheaves admits an exotic $t$-structure whose heart is known as the category of perverse sheaves.  Emerton and Kisin created a Riemann-Hilbert correspondence between the derived category of unit $F$-modules and the category of constructible $\mathbb{F}_p$-sheaves which translates the trivial $t$-structure to an exotic $t$-structure on the derived category of constructible sheaves.  It turns out the abelian categories obtained by these two methods are the same.  Therefore, it is natural to ask whether there is a purely algebraic notion of $V$-filtration for these unit $F$-modules, which are a very special type of left $\bD$-module.  The main results of this paper are\\

\textbf{Theorems \ref{mainthm} and \ref{compthm}} 
\begin{enumerate}
\item If $\mM$ is a tame $F$-crystal on $X \setminus Z$ where $Z$ is a smooth divisor, then $j_*\mM$ admits a unique V-filtration along $Z$. This filtration has the property that $Gr^0_V\mM$ is a unit $F$-crystal on $Z$.
\item Let $Sol(-)$ denote the Riemann-Hilbert functor taking unit $F$-modules to perverse sheaves on the \'etale site, $f$ a smooth divisor, and $Z$ its support. Let $\Phi_f^{un,tame}(-)$ be the tame unipotent nearby cycles functor of Grothendieck-Deligne and $\Psi_Z^{un}(\mM) = Gr_V^0 \mM$.  There is a natural isomorphism
$$\Psi^{un,tame}_f(Sol_U(-)) \simeq Sol_Z(\Psi_Z^{un}(-))$$
of sheaves on $Z_{\acute{e}t}$.
\end{enumerate}

\hspace{10pt} The nearby and vanishing cycles functors were used by Deligne in both the complex and $l \neq p$ cases to understand perverse sheaves stratified along divisors.  This approach was modified by Verdier \cite{Ve} and used to glue perverse sheaves along a divisor.  An alternative approach with similar results was provided by Beilinson \cite{Be}, whose approach reduced some of the redundancy of the gluing data by using the unipotent nearby cycles functor.  The most general approach to gluing was completed by MacPherson-Vilonen \cite{McV}.  The main criteria for $V$-filtration in the context of this paper was to recreate important instances of the gluing phenomenon.  The easiest case of gluing in characteristic zero is using the nearby cycles functor to recover the equivalence of categories between local systems on $X \setminus Z$ and representations of $\pi_1(X \setminus Z)$.  A positive characteristic version of this theorem is proven in section $5$ using the $V$-filtration developed in section $3$.  \\

\hspace{10pt} Even though a notion of Bernstein-Sato polynomial has been discovered by Musta{\c{t}}{\u{a}} in \cite{Mu}, the approach taken in this paper is an approach by flat descent.  In section two, formal properties of a $V$-filtration are explored and a descent argument is used to reduce existence to an \'etale local question.  After the formal properties are developed, the definition of tame $F$-crystals is given.  The work of Grothendieck-Murre is employed with the hypothesis of tame ramification to further reduce the question of existence to the case of a Kummer covering, which is explored in detail by \ref{starterthm}.  After discussing existence in section three, it is proven in section four that the filtration discovered recovers the (unipotent) nearby cycles functor.  As a final application in section five, there is a discussion of gluing split unit $F$-crystals on $\mathbb{A}^1$ and placing the $V$-filtration on extensions. \\

\textbf{Generalizations.} For expository purposes, this paper concentrates only on unit $F$-modules tamely ramified along smooth divisors.  There are obvious generalizations to the proofs and definitions for unit $F^n$-modules tamely ramified along normal crossings divisors.\\
 
\textbf{Acknowledgements.}  The author benefited from conversations with 
Sam Gunningham and was greatly assisted by notes from Kari Vilonen about 
the $V$-filtration in characteristic zero.  The author would like to thank 
David Nadler for suggesting to pursue the $V$-filtration in positive 
characteristic in the crystalline setting and his patience and assistance 
while the author found the correct context for it.  Lastly, he is very
grateful to Matthew Emerton for his guidance through the unit $F$-module Riemann-Hilbert correspondence, the suggestion to consider tame ramification, and supplying the author with very helpful insights on conducting research in positive characteristic geometry.\\

\section{Preliminary Information}
\hspace{10pt}In the first subsection of this paper, two fundamental theorems from algebraic geometry are recounted and strengthened to be specifically worded for the stronger hypotheses applied later in this paper.  In the definition of the objects, important facts from Grothendieck's theory of tame ramification will be used.  A subsection is provided outlining the main theorems used from the text by Grothendieck and Murre \cite{GM}.  The last subsection features a restricted case of the Emerton-Kisin Riemann-Hilbert correspondence due to Katz \cite{Ka}.  This version of the theorem will be used to show compatibility of the definitions from section $4$ under this correspondence.\\

Throughout this paper $\boldk$ will denote an algebraically closed field.

\subsection{Theorems from algebraic geometry}

\begin{thm}\label{grothzmt}\cite[8.12]{EGAIV}(Zariski's Main Theorem) Let $X$ be a normal locally Noetherian integral scheme and $Z \subset X$ a subscheme of codimension one.  If $U = X \setminus Z$ and $\pi_V:V \rightarrow U$ with $K(V) \supset K(U)$ a finite separable extension, then there exists a normal integral scheme $Y$ with maps $\pi: Y \rightarrow X$ and $\iota: V \rightarrow Y$ such that:\\
\begin{enumerate}
\item The following diagram is a commutative.
$$\xymatrix{V \ar@{->}[r]^{\iota} \ar@{->}[d]^{\pi_V}		& Y \ar@{->}[d]^{\pi} \\
            U \ar@{->}[r]	 & X  \\}$$
\item $\pi$ is finite and $\iota$ is an open immersion.
\item If $y \in \pi^{-1}(Z)$ and of codimension one, then $\pi(y)$ is also of codimension one.
\item If $X$ is finite type over $\boldk$, then so is $Y$.
\end{enumerate}
\end{thm}
\begin{proof}  Most of this statement is contained in the reference, therefore only a sketch is provided.
\begin{enumerate}
 \item $Y$ is constructed as the normalization of $X$ in $K(V)$ and the diagram follows from this construction.
 \item These statements are local so assume $X$ is affine.  Since $K(X)=K(U) \subset K(V)$ is finite separable and $X$ is Noetherian, $\mO_Y=\overline{\mO_X}^{K(V)}$ is a finitely generated module over $\mO_X$.
 \item If $y \in \pi^{-1}(Z)$ then, since $X$ is normal and any integral extension of a normal domain satisfies the going down axiom, there is some $y'$ with $y \in \overline{\{y'\}}$ such that $\pi(y')$ is codimension one.  However, $y$ has codimension one implies $y' = y$ since $Y$ is integral over $X$.
 \item Clear from 2.
\end{enumerate}
\end{proof}

\begin{thm}\cite{AK}\label{znthm}(Zariski-Nagata Purity Theorem) Let $X$ be a smooth $\boldk$-scheme and $\pi: Y \rightarrow X$ be a quasi-finite morphism which is generically \'etale.  If $\pi$ is \'{e}tale at all codimension one points of $Y$, then $\pi$ is \'{e}tale.
\end{thm}

\subsection{Tame ramification along smooth divisors}

In this section, $X$ is a locally Noetherian normal scheme.

\begin{df} Recall that an extension of discrete valuations rings $(A,\pi) \rightarrow (B,\Pi)$ is called tamely ramified if
 \begin{enumerate}
  \item $\pi B = \Pi^e$ and $e$ is invertible in $A/(\pi)$.
  \item The extension $A/(\pi) \subset B/(\Pi)$ is separable.
 \end{enumerate}
\end{df}

\begin{df}\label{tamedf}\cite[2.2.2]{GM} Let $Z \subset X$ be a closed subscheme of codimension one.  A normal covering $\pi : Y \rightarrow X$ is said to be tamely ramified with respect to $Z$ if
\begin{enumerate}
 \item $\pi$ is finite.
 \item $\pi$ is \'{e}tale over $U=X \setminus Z$.
 \item For all generic points $z \in Z$, the extension $\mO_{X,z} \subset \mO_{Y,y}$ is tamely ramified as discrete valuation rings for all $y \in \pi^{-1}(z)$.
\end{enumerate}
\end{df}

\begin{lem}\cite[2.2.8]{GM} \label{tamelem} If $Z \subset X$ is a closed subscheme of codimension one and $\pi: Y \rightarrow X$  is a finite map which is \'{e}tale over $X \setminus Z$, then the following are equivalent.
 \begin{enumerate}
  \item $\pi$ is tamely ramified with respect to $Z$.
  \item $\pi_{z} : Y \times_X Spec(\mO_{X,z}) \rightarrow Spec(\mO_{X,z})$ is tamely ramified with respect to $Z \times_X Spec(\mO_{X,z})$ for all $z \in Z$.
  \item $\pi_{\overline{z}} : Y \times_X Spec(\mO^{\acute{e}t}_{X,\overline{z}}) \rightarrow Spec(\mO^{\acute{e}t}_{X,\overline{z}})$ is tamely ramified with respect to $Z \times_X Spec(\mO^{\acute{e}t}_{X,\overline{z}})$ for all geometric points $\overline{z} \in Z$.
 \end{enumerate}
\end{lem}

\begin{rem}\cite[2.4.2-2.4.4]{GM} Let $X$ be connected and $Z \subset X$ a normal crossings divisor.  There is a natural way to generalize the construction of $\pi^{\acute{e}t}$ by considering normal covers which are tamely ramified with respect to $Z$.  In this way one constructs a profinite group, $\pi^{t}_1(X,Z)$, whose continuous actions on finite sets correspond to finite tamely ramified covers.  The group $\pi_1^{t}(X,Z)$ is called the tame fundamental group of $X$ with respect to $Z$.
\end{rem}

\begin{thm}\label{abhlem}\cite[XIII: Appendix 1]{SGA1}(Consequence of Abhyankar's lemma)
If $S$ is a scheme which is the spectrum of a strictly Henselian regular local ring and $D \subset S$ is a divisor with normal crossings then $\pi_1^{t}(S,D)$ is abelian.
\end{thm}

\begin{df} Let $S$ be a scheme, $\{a_i\}_{i \in I}$ a finite regular sequence of sections, and $\{n_i\}_{i \in I}$ a set of integers coprime to the residue characteristic of $S$.  Define a Kummer covering to be any cover isomorphic (as $S$-schemes) to the cover $$\mO_S[\{T_i\}]/(\{T_i^{n_i}-a_i\}).$$
\end{df}

The next corollary uses \ref{tamelem} to conclude that if $Z$ is the disjoint union of smooth divisors, then \'etale locally one may replace a tamely ramified cover with a Kummer covering.  The latter objects are clearly easier to understand.

\begin{cor}\cite[2.3.4]{GM}\label{abhlem2} If $\pi: Y \rightarrow X$ is tamely ramified with respect to a divisor $Z$ which is the disjoint union of smooth divisors, then for every point $x \in X$ there is an \'etale neighborhood of $x$, $U \rightarrow X$, such that $Y_U \rightarrow X_U$ is a finite disjoint union of Kummer coverings.
\end{cor}

\subsection{Unit F-modules, F-crystals, and a theorem of N. Katz}

\hspace{10pt} In the characteristic zero setting, there is a well-known special case of the Riemann-Hilbert correspondence: algebraic vector bundles with connection correspond to local systems for the analytic topology.  This correspondence is given by the functor of flat sections.  The analogue of this correspondence in characteristic $p$ is due to N. Katz.  It gives an equivalence of categories between unit $F$-crystals and \'etale local systems by taking the \'etale kernel of $1-F$.\\

In this section, $X$ is a smooth variety over $\boldk$, which is now assumed to be of characteristic $p > 0$.\\

\begin{df} A unit $F$-module is a quasi-coherent sheaf $\mM$ equipped with an isomorphism $\phi: F^*\mM \cong \mM$ where $F^*$ denotes pull back along the absolute Frobenius. There is an important natural $p$-linear map $$F_{\phi}:\mM \rightarrow F^*\mM \shortstack{$\phi$ \\ $\longrightarrow$} \mM.$$  Since $\phi$ is typically not explicitly mentioned, the symbol $F$ is implicitly understood to be $F_{\phi}$.
\end{df}
\begin{rem} Let $\bD$ be Grothendieck's ring of differential operators.  A unit $F$-module is a special type of $\bD$-module.  To equip a unit $F$-module with the structure of a left $\bD$-module, one first expresses $\bD$ in terms of the divided power filtration $\bD = \cup \bD^{(m)}$.  Through Morita theory, $\bD^{(m)}$ acts on $(F^n)^*\mM$, which is isomorphic to $\mM$ by $\phi$.  The actions obtained this way are all compatible and give $\mM$ the structure of a $\bD$-module.  See \cite{Ha} or \cite{Lyu} for more details.
\end{rem}

\begin{df} A unit $F$-module whose underlying quasi-coherent sheaf $\mM$ is $\mO_X$-coherent is called a unit $F$-crystal. 
\end{df}
\begin{thm}\label{katzthm} \cite[4.1.1]{Ka} There is an equivalence of categories between $\mathbb{F}_p$-local systems on $X_{\acute{e}t}$ ($\mathbb{F}_p$-representations of $\pi_1^{\acute{e}t}(X)$) and unit $F$-crystals on $X$.
\end{thm}
The correspondence is given in the following manner:\\

To a representation $W$, factor the representation through $G=Gal(\widetilde{X}/X)$ for $\widetilde{X}/X$ some finite Galois cover.  Equip $W \otimes_{\mathbb{F}_p} \mO_{\widetilde{X}}$ with the $p$-linear map $\widetilde{F}(w \otimes f) = w \otimes f^p$ and give this sheaf the diagonal tensor product representation structure over $G$.  Taking the $G$-invariants (which is just Galois descent), one obtains an $\mO_X$-coherent unit $F$-module on $X$.\\

The inverse functor is given by the kernel sheaf of $1-F: \mM \rightarrow \mM$ on the \'etale site.  Emerton and Kisin (\cite{EK}) generalized this functor to the category of unit $F$-modules.  Using their notation, it will be referred to as the solution function $Sol(-)$.  Notice that $Sol|_{unit F-crystals} = (1-F)^{\checkmark}$.\\

\begin{df} A unit $F$-crystal is called trivial if it is isomorphic as unit $F$-modules to a free $\mO_X$-module with the usual Frobenius.
\end{df}

The last technical notion of this section is the $F$-module pull back.

\begin{df}\cite[2.3.1]{EK} Let $\pi: Y \rightarrow X$ be a map of schemes and $\mM$ a unit $F$-module on $X$.  Ignoring a standard shift, define the unit $F$-module pull back $\pi^!\mM$ to be the sheaf of $\mO_Y$ modules $\pi^*\mM$ equipped with the diagonal action of $F$.  This notation is selected to emphasize the extra $F$-module structure.

\end{df}

\section{The V-filtration for tame $F$-crystals on $U=X \setminus Z$}

Throughout this section $X$ will denote a smooth variety over $\boldk$ and $Z$ a smooth divisor.  Also, a certain abuse of terminology in notation will be employed.\\

\textbf{Terminology.} The phrase ``unit $F$-crystal on $U$'' will mean the push forward to $X$ of a unit $F$-crystal from $U$. 

\subsection{Motivating example: Kummer coverings ramified along a smooth divisor}

\hspace{10pt}Consider the characteristic zero situation where $X = \mathbb{A}^1_{\mathbb{C}}$ and $Z$ is the origin. Let $\pi: Y_n \rightarrow X$ be the degree $n$ Kummer covering of the origin and $\mM$ an $\mO_U$-coherent $\bD_U$-module on $U$ such that $Sol(\mM)$ is a trivial local system on $V=Y \setminus \pi^{-1}(Z)$.  $\mM|_V \cong \mO^r|_V$ is not just a $\bD_V$-module but also has a compatible action of $Gal(V/U)$.  $\mM$ is determined by a complex valued representation of $Gal(V/U)$ and is obtained by Galois descent.  To create a $V$-filtration on every algebraic vector bundle with rational connection, it is enough to create a $V$-filtration on $j_*\mO|_{V}^{\chi}$ for every character $\chi$ of $Gal(V/U)$.  The formal properties of definition clarify how to discover the $V$-filtration.  $j_*\mO|_V^{Gal(V/U)} \subset K(Y)$ must inherit the induced filtration from the filtration on $K(Y)$ by the valuation determined by $\pi^{-1}(Z)$ divided by $n$.  That is, $V^{\frac{j}{n}}K(Y)$ is the subset $K(Y)$ with valuation greater than $j$.\\

\hspace{10pt} Attempting to mimic this situation in positive characteristic, this section will consider degree $n$ cyclic coverings of $U$ which are totally ramified along $Z$.  For the technical purpose of forcing the representation theory to be semisimple, the limitation to tamely ramified coverings will be imposed.  This should not be viewed as an arbitrary restriction, as wildly ramified coverings are likely to correspond to the notion of irregular $\bD$-modules in characteristic zero.  The tamely ramified coverings of that are of interest are the Kummer coverings.  It will now be shown that the characteristic zero process described above survives for Kummer coverings.  It will be generalized to the tamely ramified case in \ref{mainthm}.\\

\begin{df}\textit{The standard filtration on $\mO_V$}.

Assume that $Z$ is smooth and irreducible with generic point $\eta$.  Let $Y \rightarrow X$ be a Kummer covering.  Let $\tilde{\eta}$ be the unique point lying over $\eta$ and $v=v_{\tilde{\eta}} : K(Y) \rightarrow \mathbb{Z}$ its associated valuation.  Define $V^{\frac{j}{d}}K(Y) = \{ c | v(c) \geq j \}$, where $d = v(\eta)$ and let $j_*\mO_V \subset K(Y)$ have the induced filtration where $j: U_Y=\pi^{-1}(U) \rightarrow Y$ is the natural inclusion.
\end{df}

\begin{thm}\label{starterthm} Assume that $Z$ is smooth and irreducible with generic point $\eta$, $Y \rightarrow X$ a connected Kummer covering, and $\tilde{\eta}$ the unique point lying over $\eta$.  Let $V \rightarrow U$ be the restriction of this cover; it is an \'etale Galois cover since $\boldk$ is algebraically closed.  Let $W$ be an $\mathbb{F}_p$-representation of $Gal(V/U)$ and consider the unit $F$-module $\widetilde{M}_W = W \otimes_{\mathbb{F}_p} \mO_V$ equipped with the standard Frobenius morphism $\widetilde{F}$. Endow $j_*\widetilde{M}_W$ with the filtration induced by the standard filtration on the second factor. If $\mM_W = \widetilde{\mM}_W^{G}$ and $F =\widetilde{F}^{G}$ (see \ref{katzthm}) equipped with the descended filtration, then $F$ induces
$$\mO_Z \otimes_{\mO_{Z^{(1)}}} Gr^{\frac{j}{d}}\mM_W \cong Gr^{\frac{jp}{d}}\mM_W.$$

\end{thm}
\begin{proof}
$G$ is isomorphic to group of $d$-th roots of unity in $\overline{\mathbb{F}_p}$ since $\boldk$ is algebraically closed.  In particular, it is abelian and its representations over $\overline{\mathbb{F}_p}$ are semisimple.  For any representation $A$, denote by $A^{\chi} = \{ a \in A | \sigma(a) = \chi(\sigma)a \}$ the weight space for the character $\chi$.\\

Let $\overline{W} = W \otimes_{\mathbb{F}_p} \boldk$ and decompose $\overline{W} = \oplus_{\chi \in G^*} W^{\chi}$ into weight spaces. Then $$M_W = \widetilde{M}_W^G = \oplus_{\chi \in G^*} (W^{\chi} \otimes_{\boldk} \mO_V)^G = \oplus_{\chi \in G^*} \left(W^{\chi} \otimes_{\boldk} \mO_V^{\chi^{-1}} \right)$$

$$\Rightarrow Gr^{\frac{j}{d}}\mM_W \cong \oplus_{\chi \in G^*} W^{\chi} \otimes_{\overline{\mathbb{F}_p}} Gr^{\frac{j}{d}} \mO^{\chi^{-1}}_V$$\\

Thus the claim will follow if it is shown that the map
$$\mO_Z \otimes_{\mO_Z^{(1)}} W^{\chi} \otimes_{\boldk} Gr^{\frac{j}{d}}\mO_V^{\chi^{-1}} \rightarrow W^{\chi^p} \otimes_{\boldk} Gr^{\frac{pj}{d}}\mO_V^{\chi^{-p}}$$
is an isomorphism.  Recall that $\overline{W}$ is not just a representation of $G$ over $\boldk$, but one which came from a $\mathbb{F}_p$-vector space and hence has a compatible Frobenius action.  In particular, there is a canonical isomorphism $\boldk \otimes_{\boldk^{(1)}} W^{\chi} \rightarrow W^{\chi^p}$.  From this, it is necessary and sufficient to show that for every $\chi \in G^*$, $F$ induces an isomorphism
$$\mO_Z \otimes_{\mO_Z^{(1)}} Gr^{\frac{j}{d}}\mO_V^{\chi} \cong Gr^{\frac{pj}{d}}\mO_V^{\chi}.$$

Let $\pi$ be a uniformizing parameter of $\mO_{Y,\hat{y}}$.  Let $\sigma$ be a generator of $G$ and $\xi$ its value on $\pi$ (i.e. $\sigma(\pi) = \xi \pi$).  Let $\epsilon$ be the smallest natural number such that $\xi^{\epsilon} = \chi(\sigma)$.  First it will be proven that for every $k \geq 0$

$$V^{\frac{p^kj}{d}}\mO_V^{\chi^{p^k}} = V^{\frac{p^k(j-\epsilon)}{d}}\mO_V^{G}\pi^{p^k\epsilon}.$$

Multiplication by $\pi^{p^k\epsilon}$ is clearly an isomorphism on $\mO_V$ and it is an isomorphism between $V^{\frac{p^k(j-\epsilon)}{d}}K(Y)$ and $V^{\frac{p^kj}{d}}K(Y)$.  Thus, $V^{\frac{p^kj}{d}}\mO_V = V^{\frac{p^k(j-\epsilon)}{d}}\mO_V \pi^{p^k\epsilon}$.\\

The statement then follows because for $c=c'\pi^{p^k\epsilon} \in V^{\frac{p^kj}{d}}\mO_V$, $\sigma(c) = \chi(\sigma)c \Leftrightarrow \sigma(c'\pi^{p^k\epsilon}) = \chi(\sigma)^{p^k}c' \pi^{p^k\epsilon} \Leftrightarrow \sigma(c')\xi^{p^k\epsilon}\pi^{p^k\epsilon}=\sigma(c')\chi(\sigma)^{p^k}\pi^{p^k\epsilon}= \chi(\sigma)^{p^k}c' \pi^{p^k\epsilon} \Leftrightarrow \sigma(c') = c'$ .\\

Considering the cases $k=0$ and $k=1$, we obtain $Gr^{\frac{j}{d}}\mO_V^{\chi} = Gr^{\frac{j-\epsilon}{d}}\mO_V^{G}\pi^{\epsilon}$,  $Gr^{\frac{Pj}{d}}\mO_V^{\chi^p} = Gr^{\frac{p(j-\epsilon)}{d}}\mO_V^{G}\pi^{p\epsilon}$.  This respects the map induced by $F$ in that the diagram\\

$$\xymatrix{ V^{\frac{j-\epsilon}{d}}\mO_V^G \pi^{\epsilon} \ar@{->}[r]^{F} \ar@{->}[d]^{=} & V^{\frac{p(j-\epsilon)}{d}}\mO_V^G \pi^{p \epsilon} \ar@{->}[d]^{=} \\
V^{\frac{j}{d}}\mO_{V}^{\chi} \ar@{->}[r]^{F} & V^{\frac{pj}{d}}\mO_V^{\chi^{p}} }$$ 
commutes.  This gives the following commutative diagram,

$$\xymatrix{ Gr^{\frac{j-\epsilon}{d}}\mO_V^G \pi^{\epsilon} \ar@{->}[r]^{F} \ar@{->}[d]^{\cong} & Gr^{\frac{p(j-\epsilon)}{d}}\mO_V^G \pi^{p \epsilon} \ar@{->}[d]^{\cong} \\
Gr^{\frac{j}{d}}\mO_{V}^{\chi} \ar@{->}[r]^{F} & Gr^{\frac{pj}{d}}\mO_V^{\chi^p} }$$

Thus it is enough to show that $F$ induces this isomorphism precisely when $\chi =1$.  

Working locally, assume that $\mI_Z = (t)$.  $\mO_V^G = \mO_U$ and $\mO_U$ is given the standard Frobenius and the filtration $V^{j}\mO_U = \mI_{Z}^{j} = \mO_X t^j$.  $Gr^{j} \mO_U \cong \mO_Z$ naturally respecting $F$.  It is then clear that $F$ induces $\mO_Z \otimes_{\mO_Z^{(1)}} \mO_Z \cong \mO_Z$.
\end{proof}

\begin{rem} It is not true that the pull back of the descended filtration is the trivial filtration. 
\end{rem}

\subsection{Definition and uniqueness}

\begin{df}\label{vdef}
Let $\mM$ be a unit $F$-module.  A filtration $V^{\cdot}\mM$ is called specializing (along Z of depth $k$) if it is an exhaustive, separated, discrete and left continuous $\mathbb{Q}$-indexed descending filtration such that for every $i \in \mathbb{Q}$:
\begin{enumerate}
\item $V^{i}\mM$ is a quasi-coherent $\mO_X$-module.
\item $\mI_Z^k V^{i}\mM \subset V^{i+1}\mM$.
\item $F(V^{i}\mM) \subset V^{ip}\mM$.
\item $F$ induces $F_Z^*Gr^{i}\mM \cong Gr^{ip}\mM$ whenever $Gr^{i}\mM \neq 0$.
\end{enumerate}
\end{df}

An important property of the $V$-filtration in characteristic zero is that it is inconsequential whether the theory is developed for $\mathbb{C}$-indexed or $\mathbb{Z}$-indexed filtrations.  The next lemma gives an analogue of this phenomena in positive characteristic.

\begin{lem}\label{prelem}
Let $\mM$ be a unit $F$-module with $V^{\cdot}\mM$ and $W^{\cdot}\mM$ filtrations specializing $\mM$ (along $Z$ of depth $0$).
\begin{enumerate}
\item If $V^n\mM \subset W^n\mM$ for all integers $n$, then $V^{\cdot}\mM \subset W^{\cdot}\mM$.
\item If $V^n\mM = W^n\mM$ for all integers $n$, then $V^{\cdot}\mM = W^{\cdot}\mM$.
\end{enumerate}
\end{lem}
\begin{proof}
Let $r \in \mathbb{Q}$, $m \in V^{r}\mM$ and let $s$ be the unique rational number so that $m \in W^{s}\mM$ and $0 \neq \overline{m} \in Gr_W^{s}\mM$.  By condition $4$ of the definition, $F: Gr^{s}_W\mM \rightarrow Gr^{ps}_W\mM$ is injective so $F^n(m) \notin W^{> sp^n}\mM=\cup_{i > sp^n}W^i\mM$ for all $n$.  Yet $F^n(m) \in V^{r p^n}\mM$ for all $n$.  Thus for every $n$ there can exist no integer $t$ with the property that $ rp^n \geq t > sp^n$, as this would imply $F^n(m) \notin W^{>sp^n}\mM \supset W^t\mM \supset V^t\mM \supset V^{rp^n}\mM \ni F^n(m)$.  The property that for every $n$ no such integer exists is equivalent to the condition that $s \geq r$.  Hence $m \in W^{s}\mM \subset W^{r}\mM$.  The choice of $m$ and $r$ were arbitrary, so $W^r\mM \subset V^r\mM$ for all $r \in \mathbb{Q}$. The latter statement clearly follows from the former. 
\end{proof}

\begin{df} 
A specializing filtration is called super-specializing if it also satisfies the following properties:
\begin{enumerate}
 \item (SS1) $V^{0}\mM$ is $\mO_X$-coherent.
 \item (SS2) $\mI_Z V^{i}\mM = V^{i+1}\mM$ for all $i \in \mathbb{Q}$, $i \neq -1$.
 \item (SS3) The map multiplication by the divisor $f$ defining $Z$, $m_f: Gr^{i}_V\mM \rightarrow Gr_V^{i+1}\mM$ is an isomorphism for all $i \neq -1$.
\end{enumerate}
\end{df}

\begin{rem} The definition of specializing could use further axiomatization.  The use of depth is to preserve the property of having a specializing filtration after pulling back along flat covers.  The definition of super-specializing is too strong for the category of all unit $F$-modules when $X$ has dimension greater than one but is a correct definition for curves.   It should also be more properly called super-specializing of depth zero.  In the more general situation, one should look for filtrations which are super-specializing when pulled-back along curves.  See section five for examples of modules with high depth non-super specializing filtrations.\\
\end{rem}
\begin{pro}
If $\mM$ is a unit $F$-module with super-specializing filtration $V^{\cdot}\mM$ and specializing filtration $W^{\cdot}\mM$, then $V^{\cdot}\mM \subset W^{\cdot}\mM$.  In particular, when super-specializing filtrations exist they are unique.
\end{pro}
\begin{proof}
Let $m \in V^0\mM$ and let $l$ be the greatest integer so that $W^l\mM \supset V^0\mM$.  Such an $l$ exists since $W^{\cdot}\mM$ is exhaustive and $V^0\mM$ is $\mO_X$-coherent.  Let $s \in \mathbb{Q}$ be such that $m \in W^{s}\mM$ and $0 \neq \overline{m} \in Gr^{s}_W\mM$.  Then $F^n(m) \subset V^0\mM \subset W^l\mM$ and $F^n: Gr^{s}_W \mM \rightarrow Gr^{sp^n}_W\mM$ is injective.  Hence, $F^n(m) \notin W^{>sp^n}\mM$ for all $n$.  This can only happen if $l \leq sp^n$ for all $n$.  Hence, $s \geq 0$ and $m \in W^{0}\mM$.  As $m$ was chosen arbitrarily, $V^0\mM \subset W^0\mM$.  By property (SS1) of $V^{\cdot}\mM$, $V^{n}\mM \subset W^{n}\mM$ for all integers $n \geq 0$.\\

Let $m \in V^{-j}\mM$ with $j > 0$ and $s \in \mathbb{Q}$ so that $m \in W^{s}\mM$ and $0 \neq\overline{m} \in Gr_W^{s}\mM$.  In local coordinates, $\overline{t^j m} \neq 0 \in Gr^{s+j}_W\mM$ unless $-j = s$.  If $\overline{t^jm} \neq 0$, then $F^n(t^jm) \notin W^{>(s+j)p^n}$ by injectivity of the map $F^n: Gr^{s+j}_W\mM \rightarrow Gr^{(s+j)p^n}_W\mM$.  Yet $t^jm \in V^{0}\mM \subset W^{0}\mM$ and so it must be that $(s+j) \geq 0$.  In either situation, it follows that $s \geq -j$ and hence that $m \in W^{-j}\mM$.  As $m$ was arbitrary, $V^n\mM \subset W^n\mM$ for all $n \in \mathbb{Z}$.  Now use the previous lemma \ref{prelem}.\\
\end{proof}

\begin{lem} \label{pblem} Let $\pi: Y \rightarrow X$ be a flat cover and $E = \pi^{-1}(Z)$.  If $\mM$ is a unit $F$-module on $X$ with specializing filtration $V^{\cdot}\mM$, then $\pi^*V^{\cdot}\mM$ is a specializing filtration for $\pi^!\mM$.  If $V^{\cdot}\mM$ is super-specializing and $\pi$ is \'etale then $\pi^* V^{\cdot}\mM$ is super-specializing.
\end{lem}
\begin{proof}
First note because the map $\pi$ is flat, $\pi^*V^{\cdot}\mM$ is a subsheaf of $\pi^!\mM$.  It is obviously exhaustive, discrete, left continuous, $\mathbb{Q}$-indexed, and descending.  As $\pi$ is a flat cover, $\pi^* (\cap_r V^{r}\mM) = \cap_r \pi^* V^{r}\mM$ so it is also separated. For every $i \in \mathbb{Q}$, $\pi^*V^i\mM$ is quasi-coherent.  Since $Y$ is Noetherian there exists an integer $k_E$ such that $\mI_E^{k_E} \subset \mI_E^k$ and thus $\mI_E \pi^*V^{i}\mM \subset \pi^*\mI_Z V^{i}\mM \subset \pi^*V^{i+1}\mM$ and $F(\pi^*V^i\mM) \subset \pi^*F(V^i\mM) \subset \pi^*V^{ip}\mM$.  By flatness, there is a natural isomorphism $\pi|_{E}^*Gr^i_V\mM \cong Gr^i_{\pi^*V}\pi^!\mM$ which is compatible with the action of $F$.  The natural isomorphism $F_E^*\mO_{E^{(1)}} \cong \mO_{E}$ confirms condition $4$ of the definition is still true.\\

When $\pi$ is \'etale, $E$ is the disjoint union of smooth divisors since $Z$ is smooth. From this and working locally conditions (SS2) and (SS3) follow. Condition (SS1) is obvious.\\
\end{proof}
\begin{lem}\label{flatlem}
Let $\pi: Y \rightarrow X$ be a flat cover and $E = \pi^{-1}(Z)$.  Suppose that $\mN$ is a unit $F$-module on $Y$ with descent data over $X$ and $\mM$ the descent of $\mN$ to $X$.
\begin{enumerate}
\item If $\mN$ has a specializing filtration (along $E$) $V^{\cdot}\mN$ with compatible descent data then $\mM$ has a specializing filtration (along $Z$) given by descent.
\item If $\pi$ is \'etale and $\mN$ has a super-specializing filtration (along $E$) $V^{\cdot}\mN$ then it is automatically stable under the descent data and the descended filtration $V^{\cdot}\mM$ is also super-specializing.
\end{enumerate}
\end{lem} 
\begin{proof}\
\begin{enumerate}
\item Let $V^{\cdot}\mM$ be the descended filtration.  It clearly satisfies axioms $1$ and $2$ of \ref{vdef}.  By assumption, the descent data is compatible with $F$ so axiom $3$ is satisfied.  Axiom $4$ follows from the exactness of descent.
\item Let $p_1$ and $p_2$ be the usual projection maps $Y \times_X Y \rightarrow Y$.  By \ref{pblem}, there are unique super-specializing filtration on $p_1^!\mN$ and $p_2^!\mN$ obtained by pull-back.  The covering isomorphism $\phi: p_1^*\mN \rightarrow p_2^*\mN$ must then interchange these two filtrations since all the axioms of super-specializing filtrations are preserved under isomorphism.  Hence, there is a descent data of the filtration which is compatible with the descent data for $\mN$.\\

By the previous part of this proof, it is only left to check that the descended filtration satisfies the axioms (SS1), (SS2) and (SS3). The first is a classical result in commutative algebra (\cite[2.5.2]{EGAIV}) and the latter two follow directly from $\pi$ being unramified.
\end{enumerate}
\end{proof}

\subsection{Existence for tame $F$-crystals on $U$}
\begin{df} A unit $F$-crystal $\mM$ on $U$ (recall the terminology set at the beginning of this section) is said to be tame with respect $Z$ if for every closed point $\overline{z} \in Z$, there is a Zariski open neighborhood $W$ of $\overline{z}$ and a tamely ramified cover $Y \rightarrow W$ which trivializes $\mM|_{U \cap W}$ as a unit $F$-crystal. 
\end{df}
\subsubsection{Reduction to Kummer coverings}
\begin{lem} Any tame $F$-crystal is \'etale locally trivialized by generalized Kummer coverings.
\end{lem}
\begin{proof} Let $Y$ be a tamely ramified cover of a Zarsiki neighborhood $W$ trivializing $\mM|_{U \cap W}$ .  By \ref{abhlem2}, for each closed point $\overline{w} \in W$ there is an \'etale neighborhood $U'$ of $\overline{w}$ for which $Y_{U'} \rightarrow U'$ is a disjoint union of Kummer covering.
\end{proof}
\begin{thm} \label{redthm}If $\mM$ is a tame $F$-crystal then $\mM$ has a super-specializing V-filtration (along $Z$) if and only if for every closed point $\overline{z} \in Z$ there exists a pair $Y \rightarrow U'$ of a generalized Kummer cover trivializing $\mM_{U \times_X U'}$ and an \'etale neighborhood $U'$ of $\overline{z}$ such that $\mM|_{U'}$ has a super-specializing filtration.
\end{thm}
\begin{proof}
Necessity follows from \ref{pblem}.
To prove sufficiency, the critical observation is that condition $2$ of \ref{vdef} requires $Gr^{i}_V\mM$ is supported on $Z$.  Thus $V^{\cdot}\mM|_{U}$ (if it exists) has to be the trivial filtration $V^{\cdot}\mM = \mM$.  By hypothesis, $Z$ may be covered by a finite number of \'etale neighborhoods $U'_i$ where a super-specializing $V$-filtration exists on $\mM|_{U_i}$.  By \ref{flatlem} each $U_i$ determines a $V$-filtration on a Zariski open neighborhood.  By the unicity of super-specializing filtrations, the filtrations must agree on double and triple intersections of the $\{U_i\}$.  As the collection $\{U_i\}$ covers $Z$, the sheaf condition ensures there is a Zariski open set $W' \supset Z$ where a super-specializing $V$-filtration exists on $\mM|_{W'}$.  Hence, by the earlier observation, it is known how to extend the $V$-filtration from a Zariski neighborhood of $W$ of $\overline{z}$ to $W \cup U$.
\end{proof}

\subsubsection{Existence}

\begin{lem}
In the situation of \ref{starterthm}, the filtration created is super-specializing for the module $\mM_W$.
\end{lem}

\begin{thm}\label{mainthm}(Main Theorem) If $\mM$ is a tame $F$-crystal and $Z$ is smooth then $\mM$ admits a unique super-specializing V-filtration.
\end{thm}
\begin{proof} First use \ref{redthm} to reduce this question to a question about Kummer coverings. Now use the previous lemma.
\end{proof}

\begin{rem} One can replace smooth by normal crossings divisors by reworking \ref{starterthm}.
\end{rem}

\section{The tame nearby cycles functors}
In this section, assume that $X$ is affine for simplicity.  It will be necessary to work on the \'etale site in both the context of local systems and unit $F$-modules.  To differentiate the functors $p^*$ and $p_*$ on the \'etale site from the Zariski versions, the notation $p^*_{\acute{e}t}$ and $p_*^{\acute{e}t}$ will be used.
\subsection{Definition}
\subsubsection{The tame nearby cycles functor of Grothendieck and Deligne}

\begin{df}\cite[\'Expose I 2.7]{SGA7I}\label{delignedf} Let $f: X \rightarrow \mathbb{A}^1_{\boldk}$ be a regular function and $S$ the spectrum of the \'etale local ring of $\mathbb{A}^1_{\boldk}$ at the origin.  If $X_S = X \times_{\mathbb{A}^1_{\boldk}} S$ then $f$ defines a function $\tilde{f}: X_S \rightarrow S$.  Let $\overline{s} \in S$ be the closed point, $\eta \in S$ the generic point, $\eta^{t}$ a point defining the tame closure of $S_\eta$, and $S^{t}$ the normalization of $S$ in $\eta^{t}$.  Let $i$ be the inclusion of the special fiber into $X^t :=X_S \times_S S^t$, $j :U_S^t := U_S \times_S S_{\eta^t} \rightarrow X$, and $g:{X_S}^t \rightarrow X$ the natural map .   For an \'etale local system on $\mL$ on $X$, define the (underived) tame nearby cycles functor by

$$\Psi_f^{tame}(\mL) = i^*_{\acute{e}t}j^{\acute{e}t}_*j^*_{\acute{e}t}g^*_{\acute{e}t}\mL.$$
\end{df}

\begin{rem} It is important to remember that $\Psi^{tame}_f(\mL)$ is not just a sheaf on the special fiber, but that it is also equipped with a continuous action of the group $Gal(\eta^t/\eta)$.
\end{rem}

\begin{df} Using this action, the hypothesis of \ref{delignedf}, and following Beilinson's presentation \cite{Be}, define the tame unipotent nearby cycles functor $\Psi^{un,tame}_f(-)$ to be the composition of $\Psi_f$ with the functor of invariants $(-)^{Gal(\eta^t/\eta)}$.
\end{df}
\subsubsection{Unipotent nearby cycles functor for tame $F$-crystals}
\begin{df}(The unipotent nearby cycles functor for tame $F$-crystals)
Let $\mM$ be a tame local system on $U$, equip $\mM$ with its unique super-specializing $V$-filtration. Let $i: Z \rightarrow X$ then $\mM$ with super-specializing $V$-filtration $V^{\cdot}\mM$.  Define the nearby cycles of $\mM$ as

$$\Psi^{un}_{Z}(\mM) = Gr^{0}(\mM) \shortstack[c]{$F$ \\ $\longrightarrow$} Gr^{0}(\mM).$$

\end{df}

\subsection{Compatibility of definitions under the Riemann-Hilbert correspondence}
\begin{thm}\label{compthm} Suppose that $X$ is affine and $Z$ is smooth and defined by the function $f:X \rightarrow \mathbb{A}^1$.  Let $Sol(-)$ be the functor described in \ref{katzthm} taking unit $F$-crystals to local systems on the \'etale site of $X$. There is a natural isomorphism
$$\Psi_f^{tame,un}(Sol_U(-)) \simeq Sol_Z(\Psi_Z^{un}(-))$$

of sheaves on $Z_{\acute{e}t}$.
\end{thm}
\begin{proof} To ease notations, replace $\Psi_f^{tame,un}$ and $\Psi_Z^{un}$ by $\Psi_f$ and $\Psi_Z$.  The functor $i^*_{\acute{e}t}j^{\acute{e}t}_*j^*_{\acute{e}t}g^*_{\acute{e}t}(-)^{Gal(\eta^t/\eta)}$ is left exact and $Sol(\mM) = Ker(\mM \shortstack{$1-F$ \\ $\longrightarrow$} \mM)$ so there is a natural isomorphism
$$\Psi_f(Sol(\mM)) = i^*_{\acute{e}t}j^{\acute{e}t}_*j^*_{\acute{e}t}g^*_{\acute{e}t}(Sol(\mM)))^{un} \cong Kern(i^*_{\acute{e}t}j^{\acute{e}t}_*j^*_{\acute{e}t}g^*_{\acute{e}t}\mM \shortstack{$1-F$ \\ $\longrightarrow $} i^*_{\acute{e}t}j^{\acute{e}t}_*j^*_{\acute{e}t}g^*_{\acute{e}t}\mM)^{Gal(\eta^t/\eta)}.$$ 
The question reduces down to the two following questions.

\begin{enumerate}
 \item The natural map $\Psi_f(Sol(\mM)) \rightarrow i^*_{\acute{e}t}j^{\acute{e}t}_*j^*_{\acute{e}t}g^*_{\acute{e}t}\mM$ has image in the subsheaf $i^*_{\acute{e}t}j^{\acute{e}t}_*j^*_{\acute{e}t}g^*_{\acute{e}t}V^{0}\mM$.
 \item The induced map $\Psi_f(Sol(\mM)) \rightarrow Gr_{i^*_{\acute{e}t}j^{\acute{e}t}_*j^*_{\acute{e}t}g^*_{\acute{e}t}V^{\cdot}\mM}^0\mM$ is an isomorphism of sheaves.
\end{enumerate}

\begin{enumerate}
\item
By the universal properties of sheafification, it is only necessary to check these conditions locally on the level of presheaves on the \'etale site. Let $\overline{z}$ be a closed point of $Z$ and $W$ an \'etale neighborhood which makes $\mM|_W$ trivialized by a Kummer covering.  By the definition of the pull back it suffices to show $Ker(1-F)(W) \subset V^0\mM(W)$.  This is obvious by the constructions in the proof of \ref{starterthm}.  In fact, one can repeat this with any \'etale neighborhood $W' \rightarrow W$ of $\overline{z}$, so as sheaves on the \'etale site of $X$, the natural map $Ker(1-F)|_W \rightarrow \mM|_W$ has presheaf image inside $V^0\mM|_W$.\\  Recall that $V^0\mM|_U=\mM|_U$ so $Ker(1-F)|_U \subset V^0\mM|_U$.  Thus it has been shown that $Kern(1-F) \subset V^0\mM$ as sheaves.\\

\item Let $\{W_{i}\}$ be a finite collection of \'etale neighborhoods $W_i \rightarrow X$ covering $Z$, which make $\mM|_{W_i}$ trivialized by a Kummer covering $Y_i \rightarrow W_i$.  For each $i$, let $l_i$ be the degree of the extension $Y_i \rightarrow W_i$.  Let $S_l$ be the $l$-th degree Kummer extension of $S$ where $l=\prod_i l_i$ and consider the following diagram.
$$\xymatrix{{W_i}\times_X X_S^t  \ar@{->}[d] \ar@{<-}[r]^{j_t^i} & W_i \times {U_S}^t \ar@{->}[d]^{\pi_l} \\
W_i \times_X X_S \times_S S_l \ar@{->}[d]^{g_l^i} \ar@{<-}[r]^{j_l^i} & W_i \times_X U_S \times_S S_l \\
X & \\}$$

By construction, $(j^i_l)^!(g^i_l)^!\mM$ is a trivial unit $F$-crystal.  By the construction given in \ref{starterthm} the natural map of sheaves, $((g^i_l)^*_{\acute{e}t}Ker(1-F))^{Gal(\eta_l/\eta)} \rightarrow Gr^0_{(g^i_l)^*_{\acute{e}t}V}(g^i_l)^*_{\acute{e}t}\mM$ is an isomorphism over $W_i \times_X Z \times_X X_S \times_S S_l$ and this is its support. There is a natural isomorphism $(j^i_t)^{\acute{e}t}_*(\pi_l)^*_{\acute{e}t}(j^i_l)^*_{\acute{e}t}(g^i_l)^*_{\acute{e}t}(-)$ with $j_*^{\acute{e}t}j^*_{\acute{e}t}g^*_{\acute{e}t}(-)|_{W_i\times_X X^t_S}$ By choice of $l$, the action of $Gal(\eta^t/\eta)$ on $\Psi_f$ factors through $Gal(\eta_l/\eta)$ so taking unipotent parts upstairs and downstairs are the same.  Thus, the natural map is an isomorphism of sheaves $(\iota^i)^*(j_*j^*g^*Ker(1-F)^{un}|_{W_i \times_X X_S^t} \rightarrow (\iota^i)^*(j_*j^*g^*Gr^0_V\mM)|_{W_i \times_X X_S^t}$ where $\iota^i$ is the inclusion of the special fiber into $W_i \times X \times X_S^t$.  It has been shown that the natural map is an isomorphism locally so it is an isomorphism globally.

\end{enumerate}

\end{proof}

\begin{rem} For ease of exposition only the unipotent part of the nearby cycles functor was used.  One can define nearby cycles in general and this equivalence will still hold. However, since $Gr^{i}\mM$ is not a unit $F$-module, one has to use the categorical techniques employed in the next section to make sense of the functor $Sol(-)$ on this object.
\end{rem}

\section{Applications to $\mathbb{A}^1_{\boldk}$}
\hspace{10pt} Now some applications to the case when $X=\mathbb{A}^1_{\boldk}$ and $Z$ is the origin are explored to indicate that the $V$-filtration created in section three has some of typical properties of the characteristic zero situation.  Namely, it will be shown how to use the $V$-filtration to recover information without passing through the Riemann-Hilbert correspondence.  A typical application in characteristic zero is to try to recover representations of the fundamental group.  This task is too difficult in our setting, as the $V$-filtration only keeps track information \'etale locally near $Z$.  While every tame unit $F$-module is \'etale locally trivialized by a Kummer extension, it is not globally trivialized by a Kummer extension.  However, in \ref{tamethm} it is shown that one can use the information provided by the nearby cycles functor to recover representations which are trivialized by a Kummer covering.  This leads naturally into considering a gluing construction as it permits one to realize they can recover information very near $Z$ by the nearby cycles functor and the sheaf away from $Z$ is part of the gluing data. \\

\hspace{10pt} There is a technical challenge to this process.  First, the associated graded gives data $F: \mN \rightarrow \mN'$ with the property that $F^*\mN \rightarrow \mN'$ is an isomorphism but this is not exactly a unit $F$-module.  In the first section, some technical categorical equivalences are explored to make the data extraction of data in section two rigorous.  For a gluing theorem, one wants to put a $V$-filtration on the entire category on unit $F$-modules which is easiest to do for for extensions which are \'etale locally split.  This is the approach taken in section three in \ref{gluethm}.  In the last section, non-trivial extensions possessing $V$-filtrations are explored.\\
 
\subsection{Preliminary categorical equivalences}
For a topological group, $G$, denote the set of continuous characters into $\boldk$ (with the discrete topology) by $G^{\checkmark}$.  A map of vector spaces graded by an abelian group $A$ is said to be $p$-graded if it maps the $a^{th}$ component to the $p a^{th}$ space.

\begin{df}\label{catdf} Let $G$ be a profinite group.  Let $\mmC(G)$ be the category with objects ($V$ \shortstack[c]{$\tau$ \\ $\rightarrow$} $W$) such that:
\begin{enumerate}
\item $V$ is a $G^{\checkmark}$-graded finite dimensional $\boldk$-vector space.
\item $W$ is $G^{\checkmark}$-graded finite dimensional $\boldk$-vector space.
\item $\tau$ is a $p$-graded $p$-linear map.
\item The natural map $\tilde{\tau}: \boldk \otimes_{\boldk^{(1)}} V^{(1)} \rightarrow W$ is an isomorphism.
\end{enumerate}
A morphism between $\tau: V \rightarrow W$ and $\tau': V' \rightarrow W'$ in this category is given by a pair $(f,g)$ of graded maps such that $g \circ \tau = \tau' \circ f$.\\
\end{df}

\begin{thm}\label{catequiv} Suppose that $G$ is abelian and every finite quotient of $G$ has order coprime to $p$. Then there is an equivalence of categories,
$$Rep_{cont}(G, \mathbb{F}_p) \cong \mmC(G)$$
where $Rep_{cont}(G, \mathbb{F}_p)$ is the category of continuous finite dimensional representations over $\mathbb{F}_p$.
\end{thm}

\begin{df} Define a functor $\mathbb{F}: Rep_{cont}(G, \mathbb{F}_p) \rightarrow \mmC(G)$ by the rule:\\
For every object $V$ of $Rep_{cont}(G, \mathbb{F}_p)$,
$\mathbb{F}(V) = \{V \otimes_{\mathbb{F}_p} \boldk \shortstack[c]{$\tau_V$ \\ $\rightarrow$} V \otimes_{\mathbb{F}_p} \boldk \}$ where $\tau_V(v \otimes c) = v \otimes c^p$ and
$V \otimes_{\mathbb{F}_p} \boldk$ is given a $G^{\checkmark}$-grading by decomposition of the action $G$.
For every morphism $f: V \rightarrow V'$, $\mathbb{F}(f) = (f \otimes id, f \otimes id)$.
\end{df}

\begin{pro}
The functor $\mathbb{F}$ is well-defined.
\end{pro}
\begin{proof}
First it will be checked the functor is well-defined on objects.
Let $V$ be a continuous $G$ representation over $\mathbb{F}_p$.

By the hypothesis on $G$ that every finite quotient has order coprime to $p$, the $G$ representation $V \otimes_{\mathbb{F}_p} \boldk$ will be semisimple.  Moreover, as $G$ is abelian, $V \otimes_{\mathbb{F}_p} \boldk$ will be a finite direct sum of character spaces.  In particular, it has a  grading by elements of the abelian group $G_1^{\checkmark}$. Thus $\mathbb{F}(V)$ satisfies conditions $1$ and $2$ of \ref{catdf}.\\

Axiom $4$ of \ref{catdf} is clearly satisfied, it remains to show axiom $3$.\\

$\tau_V$ is clearly a $p$-linear map.  It is also a map of $G$ representations.  Suppose that $x \in V \otimes_{\mathbb{F}_p} \boldk$ is in the weight space of weight $\chi \in G^{\checkmark}$.  For any $g \in G$
\begin{eqnarray*}
g \tau_V(x) &=& \tau_V (g x)\\
							&=& \tau_V (\chi(g) x)\\
							&=& \chi(g)^p \tau_V(x),
\end{eqnarray*}

showing that $\tau_V$ is $p$-graded.
\end{proof}

\begin{df}\label{Gdf} Define a functor $\mathbb{G}: \mmC(G) \rightarrow Rep_{cont}(G, \mathbb{F}_p)$ by the following rule:\\
For every object $V \shortstack[c]{$\tau$ \\ $\rightarrow$} W$, $\mathbb{G}(\tau) = Ker(id_V - m \circ \tilde{\tau}^{-1} \circ \tau)$ with reference to the diagram,
$$ V \shortstack[c]{$\tau$ \\ $\longrightarrow$} W \shortstack[c]{$\tilde{\tau}^{-1}$ \\ $\longrightarrow$} \boldk \otimes_{\boldk^{(1)}} V \shortstack[c]{$m$  \\ $\cong$ \\ $\longrightarrow$} V.$$
$V$ is given a $G$ action by specifying that $G$ acts on the $\chi \in G^{\checkmark}$ graded piece via the character $\chi$.\\
For every morphism $(f,g)$ set $\mathbb{G}((f,g)) = f|_{\mathbb{G}(\tau)}$.
\end{df}

\begin{pro} The functor $\mathbb{G}$ is well-defined.
\end{pro}
\begin{proof}
The only thing to show is that $m \circ \tilde{\tau}^{-1} \circ \tau$ respects the action of $G$. $\tau$ is $p$-graded so it respects the $G$-action.    The multiplication map also respects the $G$ action.  It is enough to show that $\tilde{\tau}$ respects the $G$ action.  It will respect the $G$ action since $\tau$ did.
\end{proof}

\begin{proof}(of \ref{catequiv})
First it is shown that $Id \simeq \mathbb{G} \circ \mathbb{F}$.\\

Fix $V$ a $G$ representation over $\mathbb{F}_p$.  There is a map $n_V: V \rightarrow V \otimes_{\mathbb{F}_p} \boldk$ given by $v \mapsto v \otimes 1$.

One needs to check that image of $n_V$ is inside $\mathbb{G}(\mathbb{F}(V))$.  This will be shown by showing that $m \circ \tilde{\tau_V}^{-1} \circ \tau_V = \tau_V$.

\begin{eqnarray*}
(m \circ \tilde{\tau_V}^{-1} \circ \tau_V)(v) &=& m(\tilde{\tau_V}^{-1}(1 \cdot \tau(v))\\
																							&=& m(1 \otimes \tau(v))\\
																							&=& \tau(v)
\end{eqnarray*}

Thus $n_V: V \rightarrow (\mathbb{G} \circ \mathbb{F})(V)$ is a well-defined map of sets.  It is clearly a well-defined map of $G$ representations over $\mathbb{F}_p$.\\

It is easy to see that $n_V$ is an isomorphism.\\

It remains to show that $n_V$ satisfies the naturality condition to be a natural isomorphism.  The following diagram commutes for any map of representations $f: V \rightarrow V'$.

$$\xymatrix@1{V \ar@{->}[r]^-{n_V} \ar@{->}[d]^-{f}			& (\mathbb{G} \circ \mathbb{F})(V) \ar@{->}[d]^{(\mathbb{G} \circ \mathbb{F})(f)}  \ar@{->}[r]^-{=} & Ker(id-\tau_V) \ar@{->}[d]^-{(f \otimes id)|}\\
            V' 	\ar@{->}[r]^-{n_{V'}} & (\mathbb{G} \circ \mathbb{F})(V') \ar@{->}[r]^-{=} & Ker(id - \tau_{V'})}$$

This completes the proof of $\mathbb{G} \circ \mathbb{F} \simeq 1$.\\
            
The case of $\mathbb{F} \circ \mathbb{G} \simeq Id$ is now considered.\\

For an object $V \shortstack[c]{ $\tau$ \\ $\rightarrow$ } W$, define $n_{\tau}$ by the following commutative diagram:\\

$$\xymatrix@1{ (\mathbb{F} \circ \mathbb{G})(\tau) \otimes_{\mathbb{F}_p} \boldk \ar@{->}[d]^-{\tau_{\mathbb{F}(\tau)}} \ar@{->}[rr]^-{m} &  & V \ar@{->}[d]^-{\tau} \\
(\mathbb{F} \circ \mathbb{G})(\tau) \otimes_{\mathbb{F}_p} \boldk \ar@{->}[rr]^-{g_{\tau}} \ar@{->}[dr]^-{m} &  & W\\
& V \ar@{->}[ur]^-{\cong}_-{(m \circ \tilde{\tau}^{-1} \circ \tau)^{-1}} &}$$

That is, define $n_\tau=(m, g_{\tau})$.\\

It is clear that the maps $m$ and $g_{\tau}$ respect the action of $G_1 \times G_2$, thus they are graded maps of $G_2$ representations.

One can observe that by Katz's theorem applied to the point $Spec(\boldk)$, the map $m$ is an isomorphism.  By the diagram, $g_{\tau}$ is an isomorphism.\\ 

The naturality condition is the last item to check.  For this, consider the following diagram which is easily confirmed as commutative for any morphism $(f,g)$ of $\mmC(G)$.

$$\xymatrix{  (\mathbb{F} \circ \mathbb{G})(\tau) \otimes_{\mathbb{F}_p} \boldk \ar@{->}[dr]^-{\tau_{\mathbb{F}(\tau)}} \ar@{->}[rr]^-{m} \ar@{->}[ddd]^-{(\mathbb{G} \circ \mathbb{F})((f,g))} &  & V \ar@{->}[rr]^-{\tau} \ar@{-}[d]|!{"2,2";"1,5"}\hole & & W \ar@{->}[ddd]^-{g} \\
& (\mathbb{F} \circ \mathbb{G})(\tau) \otimes_{\mathbb{F}_p} \boldk \ar@{->}[urrr]^-{g_{\tau}} \ar@{->}[rr]^-{m} \ar@{->}[ddd]^-{(\mathbb{G} \circ \mathbb{F})((f,g))} & \ar@{->}[dd]^-{f} & V \ar@{->}[ur]^-{\cong}_-{(m \circ \tilde{\tau}^{-1} \circ \tau)^{-1}} \ar@{->}[ddd]^-{f} &\\
& & & &\\
  (\mathbb{F} \circ \mathbb{G})(\tau') \otimes_{\mathbb{F}_p} \boldk \ar@{->}[dr]^-{\tau_{\mathbb{F}(\tau')}} \ar@{->}'[r][rr]^-{m} &  & V' \ar@{->}'[r][rr]^-{\tau'} & & W' \\
& (\mathbb{F} \circ \mathbb{G})(\tau') \otimes_{\mathbb{F}_p} \boldk \ar@{->}[urrr]^-{g_{\tau'}}|!{"2,4";"5,4"}\hole \ar@{->}[rr]^-{m} &  & V' \ar@{->}[ur]^-{\cong}_-{(m \circ \tilde{\tau'}^{-1} \circ \tau')^{-1}} &\\}
$$

\end{proof}

\subsection{Equivalence of categories $Rep_{\mathbb{F}_p}(\pi_1^{K}(X,Z))$ and Kummer $F$-crystals}

\begin{df} If $\mM$ is a tame $F$-crystal then define the nearby cycles functor as
$$\Psi_Z(\mM) = Gr^{[0,1)}i^*_{\acute{e}t}\mM \shortstack{$F$ \\ $\longrightarrow$} Gr^{p[0,1)}i^*_{\acute{e}t}\mM .$$
\end{df}
\begin{df} Let $\pi_1^{K}(X,Z)$ denote the profinite quotient subgroup of $\pi_1^{tame}(X,Z)$ obtained by taking the inverse limit of Kummer extensions of $X$ ramified only at $Z$. There is an obvious natural isomorphism $$\pi_1^{K}(X,Z) \cong \prod_{p \nmid l} \mathbb{Z}_l.$$  A unit $F$-crystal will be called Kummer tame if it is trivialized by a Kummer covering of $X$.
\end{df}
\begin{thm}\label{tamethm}

$$\{\text{Kummer $F$-crystals} \} \cong Rep_{cont}(\pi_1^{K}(X,Z), \mathbb{F}_p)$$

Given by $\mathbb{G} \circ \Psi_{Z}$ where $\mathbb{G}$ is as in \ref{Gdf}.
\end{thm}
\begin{proof}
The functor is clearly fully faithful and essentially surjective.
\end{proof}

\subsection{Quiver/Gluing description of locally split finitely generated tame unit $F$-modules on $\mathbb{A}^1_{\boldk}$.}

\hspace{10pt} As mentioned in the previous discussion, the category of tame local systems is too complicated to be expressed as a finite dimensional $\boldk$-vector space with the action of an operator.  Likewise, the category of tame finitely generated $F$-modules is too complex to be described by a quiver of finite dimensional vector spaces as it is in the characteristic $0$ case.  However, a theorem is now provided which allows one to recover information from the nearby and vanishing cycles quiver if some of the global data is remembered as in the classical gluing constructions.

\begin{df} With the notation of \ref{delignedf}, a unit $F$-module on $X$, $\mM$, is said to be split near $Z$ if its extension class $[\mM|_{S^{t}}]=0 \in Ext^1_{\mO_{S^t}[F]}(Im(\alpha), Ker(\alpha))$ where $\alpha : \mM \rightarrow j_*j^!\mM$ is the natural adjunction map.  It will be called tame split if on the largest subset which it is $\mO_X$-coherent, it determines a tamely ramified local system near $Z$ and if it is split near $Z$.
\end{df}
\begin{pro} Every object in the category of tame split unit $F$-modules on $\mathbb{A}^1_{\boldk}$ admits a unique $V$-filtration.
\end{pro}
\begin{proof}
By previous considerations, one only needs to show it exists in a Zariski neighborhood of $Z$.  Take this Zariski neighborhood to be $Z$ union with the largest open set on which $\mM$ is $\mO_X$-coherent. Using the splitting and tameness conditions, in an \'etale neighborhood of $Z$, $\mM \cong Im(\alpha) \oplus Ker(\alpha)$.  The filtration on the local system $Im(\alpha)$ has already been discovered and define the filtration on $Ker(\alpha)$ by setting $V^{i}Ker(\alpha)=\{ k  | t^i k =0\}$.  The direct sum filtration is super-specializing so by \ref{flatlem} it determines a $V$-filtration in a Zariski neighborhood of $Z$.
\end{proof}

\begin{df} Define the (unipotent) vanishing cycles sheaf to be $Gr^{-1}(\mM) \shortstack{$F$\\ $\rightarrow$} Gr^{-p}(\mM)$.  There is a natural map $(t,t^p)$ from this object to $\Psi_Z^{un}(\mM)$
\end{df}

\begin{df}
Let $\mG$ be the category of with objects triples $(\mN, V \rightarrow W, a)$ where $\mN$ is a finitely generated unit $F$-module on $U$, $V \rightarrow W$ is a $p$-linear map and $a: (V \rightarrow W) \rightarrow \Psi_Z^{un}(j_*\mN)$.  This category is the category of gluing data.
\end{df}
\begin{thm}\label{gluethm}There is an equivalence of categories

$$\{\text{Tame split l.f.g.u $F$-modules}\} \simeq \mG.$$
\end{thm}
\begin{proof}
 The functor $\mM \mapsto (\mM|_{U}, \Psi_Z^{un}(\mM), (t,t^p))$ is fully faithful and essentially surjective.
\end{proof}

\subsection{The $V$-filtration for some non-split extensions}

In this section $X$ is dimension one, $Z$ is a point, and $\Delta_Z = \mO_X|_U/\mO_X$ denotes the unit $F$-module of delta functions on $Z$

\begin{thm}\ \label{curvesthm} \begin{enumerate} 

\item The $\boldk$-vector space $Ext^1_{unit F}(j_*\mO_X, \Delta_Z)$ has a basis generated by classes $[\frac{1}{t^{n+1}}]$ for $n \in \mathbb{N}$. 
\item Each extension extension $\mM_c$ determined by a class $c(t) \in Ext^1_{unit F}$ admits a super-specializing $V$-filtration when $p \nmid v_t(c)+1$.  This filtration makes the sequence $$0 \rightarrow \Delta_Z \rightarrow \mM_{c} \rightarrow j_*\mO_U \rightarrow 0$$
exact after shifting the filtration on $j_*\mO_U$ by $\frac{-n}{p}$.
\item The extensions determined by the classes $[\frac{1}{t^{lp^k+1}}]$ 
($p \nmid l, k \geq 1$) admit a super-specializing filtration for which this becomes a filtered extension after shifting the filtration on $j_*\mO_U$.  It also admits a separate specializing filtrations of depth $0, p, ..., p^k$.  This can be generalized to classes $c(t)$ with $v_t(c)+1 = p^kl$.
\end{enumerate}
\end{thm}

\begin{rem} For each class $c$ with $p \nmid v_t(c)+1$ the extension $\mM_c$ is a type of maximal extension in the category of super-specializable modules.  The shift in the filtration appears to account for the Tate twist which appears for perverse sheaves in the $l \neq p$ situation.  See \cite{Be}.
\end{rem}

\begin{pro} The sheaf $\mM_{c}$ is the $\mO_X$-module $j_*\mO_U \oplus \Delta_Z$ equipped with the Frobenius $F(f,g)=(f^p, tf^pc(t) + g^p)$.  Equipped with the obvious inclusion and projection maps it is an extension of unit $F$-modules.
\end{pro}

\begin{proof}(Proof of \ref{curvesthm})
\begin{enumerate}
\item Follows directly from the free $\mO_X[F]$ resolutions of $j_*\mO_U$ given by the root morphism in \cite{Lyu} and the associated complex from \cite{EK} (5.3.3).  Equivantly, it is the free resolution given by $\mO_X[F]$ acting on $t^{-1}$ with kernel $1-t^{p-1}F$.\\

\item Set $n=v_t(c(t))+1$ and using the description of the previous proposition, define a filtration by $(f, \overline{g}) \in V^{j}$ if only if $min\{v_t(f)-n/p, v_t(\overline{g})\} \geq j$ for all $j \in \mathbb{Q}$ and $j \leq 0$.  For $j \geq 0$ one requires $g=0$ and $v_t(f) - n/p \geq j$.  Axioms $1$ and $2$ of a specializing filtration are clear.  Axiom $3$ follows through from a case by case analysis.  Moreover, the non-zero components of the associated graded module are $Gr^{r}\mM_{c}$ with $r = i - n/p$ or $r=i$ for any integer $i$.  In the non-integer $r=i-n/p$ case it is a single vector space generated by $(t^{i},0)$ and in the integer case $r=-i \leq 1$ it is generated by $(0,\frac{1}{t^i})$.  It is easy to confirm axiom $4$ and the axioms of super-specialization.\\

\item Only a proof is sketched and for the case $k=1$.  Give the extension the natural grading determined by setting $(t^i,-\frac{t^i}{t^l})$ to have degree $-\frac{l}{p}$ and $(0,t^i)$ to have degree $i$.  The key observation is that while both $(1,0)$ and $(0,\frac{1}{t^l})$ in the associated graded map to $(0, \frac{1}{t^{lp}})$ under $F$, they are equal in the associated graded and are mapped to by $(1,-\frac{1}{t^l})$ under $F$.  The second statement is obvious after realization that the extension determined by $[\frac{1}{t^{lp^k+1}}]$ is the $k$-th Frobenius pullback of the extension determined by $[\frac{1}{t^{l+1}}]$  \\
\end{enumerate}
\end{proof}

\bibliographystyle{alpha}
\bibliography{unitF}

\end{document}